\documentclass{amsart}[12pt]
\usepackage{amsmath, amsthm, amscd, amsfonts,}
\usepackage{graphicx,float}

\usepackage{amssymb}
\usepackage{pb-diagram}
\usepackage{lamsarrow}
\usepackage{pb-lams}
\usepackage[all,cmtip]{xy}

\newcommand{\PP}{{\mathbb{P}}}

\newcommand{\ga}{\alpha}

\newcommand{\gk}{\kappa}

\newcommand{\gt}{\tau}

\newcommand{\func }{\mathord{:}}

\newcommand{\restricted}{\mathord{\restriction}}

\newcommand{\ordered}[1]{\ensuremath{\langle #1 \rangle}}

\newcommand{\ordof}[2]{\ensuremath{\ordered{ #1 \mid #2 }}}

\DeclareMathOperator{\len}{l}

\DeclareMathOperator{\dom}{dom}
\DeclareMathOperator{\cf}{cf}
\DeclareMathOperator{\Add}{Add}

\DeclareMathOperator{\supp}{supp}
\DeclareMathOperator{\Ult}{Ult}
\DeclareMathOperator{\limdir}{lim\ dir}
\DeclareMathOperator{\id}{id}

\DeclareMathOperator{\inacc}{inacc}

\newcommand{\Es}{{\ensuremath{\bar{E}}\/}}

\newcommand{\VS}{V^*}
\newcommand{\MS}{{M^*}}

\newcommand{\NS}{{N^*}}

\newcommand{\MSt}{{M^*_\gt}}
\newcommand{\Mt}{{M_\gt}}

\newcommand{\ME}{{M_{\Es}}}

\newcommand{\MSE}{{M^*_{\Es}}}

\def\MPB{{\mathbb{P}}}

\def\k{\kappa}

\def\l{\lambda}

\def\a{\alpha}

\def\b{\beta}

\newtheorem{theorem}{Theorem}[section]
\newtheorem{lemma}[theorem]{Lemma}

\newtheorem{corollary}[theorem]{Corollary}
\newtheorem{remark}[theorem]{Remark}
\newtheorem{claim}[theorem]{Claim}

\numberwithin{equation}{section}

\def\l{\lambda}

\def\rmark{\mbox{$\rm\bf\rule{0.06em}{1.45ex}\kern-0.05em R$}}
\def\pmark{\mbox{$\rm\bf\rule{0.06em}{1.45ex}\kern-0.05em P$}}
\def\nmark{\mbox{$\rm\bf\rule{0.06em}{1.45ex}\kern-0.05em N$}}
\def\vdash{\mbox{$\rm\| \kern-0.13em -$}}

\def\l{\lambda}

\def\rmark{\mbox{$\rm\bf\rule{0.06em}{1.45ex}\kern-0.05em R$}}
\def\pmark{\mbox{$\rm\bf\rule{0.06em}{1.45ex}\kern-0.05em P$}}
\def\nmark{\mbox{$\rm\bf\rule{0.06em}{1.45ex}\kern-0.05em N$}}
\def\vdash{\mbox{$\rm\| \kern-0.13em -$}}
\newcommand{\lusim}[1]{\smash{\underset{\raisebox{1.2pt}[0cm][0cm]{$\sim$}}
{{#1}}}}

\begin{document}

\title[On Foreman's maximality principle]{On Foreman's maximality principle}

\author[Mohammad Golshani and Yair Hayut]{Mohammad
  Golshani and Yair Hayut}

\thanks{The first author's research was in part supported by a grant from IPM (No. 91030417).}
%{ \\ Department of Mathematics\\  Shahid Bahonar University of Kerman, Kerman, Iran}

%\subjclass[2010]{03E35}

%\keywords{ Multiplication mudule, prime submodule, Strongly prime submodule, valuation,
%fractional submodule, pseudo-valuation module}
\begin{abstract}
In this paper we consider  Foreman's maximality principle, which  says that any non-trivial forcing notion either adds a new real or collapses some cardinals. We prove the consistency of some of its consequences. We observe that it is consistent that every $c.c.c.$ forcing adds a real and that for every uncountable regular cardinal $\kappa$, every $\kappa$-closed forcing of size $2^{<\kappa}$ collapses some cardinal.
\end{abstract}
\maketitle

\section{introduction}
Foreman's maximality principle \cite{foreman-magidor-shelah} says that any non-trivial forcing notion either adds a new real or collapses some cardinal.  The consistency of this principle was asked by Foreman-Magidor-Shelah \cite{foreman-magidor-shelah}, who showed  that if $0^\sharp$ exists, then any non-trivial constructible forcing notion adds a new real over $V$ (see also \cite{stanley}, where a generalization of this result is proved).

 In this paper we consider the following two consequences of Foreman's maximality principle, and prove some consistency results related to them:
\begin{enumerate}
\item Any non-trivial $c.c.c.$ forcing notion adds a new real.

\item For every uncountable regular cardinal $\kappa$, every $\kappa$-closed forcing of size $2^{<\kappa}$ collapses some cardinal.
\end{enumerate}

We show that statement $(1)$ is equivalent to the assertion ``there are no Souslin trees'', and hence by  Solovay-Tennenbaum \cite{solovay-tennenbaum},  it is consistent that all non-trivial $c.c.c.$ forcing notion add a new real.

We also consider statement $(2)$, and prove that it is consistent, relative to the existence of a strong cardinal, that for all uncountable cardinals $\kappa,$ all $\k$-closed (and in fact all $\k$-strategically closed) forcing notions of size $\leq 2^{<\kappa}$ collapse $2^{<\k}.$ In such a model $GCH$ must fail everywhere, and hence we need some large cardinals to get the result. In the opposite direction, we build, assuming some large cardinals, a model in which $GCH$ fails everywhere, but for each infinite cardinal $\k,$ there exists a $\k^+$-closed forcing notion of size $2^\k$ which preserves all cardinals. Our work extends an earlier work of Foreman and Woodin \cite{foreman-woodin} by reducing their use of a supercompact cardinal to a strong cardinal.

In order to avoid trivialities, in this paper, the phrase ``forcing notion" is used only for separative non-atomic forcing notions.

\section{Consistency of any $c.c.c.$ forcing notion adding a new real}

In this section we consider statement $(1)$, and prove its consistency.

\begin{theorem}
Souslin hypothesis (SH) holds iff any non-trivial $c.c.c.$ forcing notion adds a new real.
\end{theorem}
\begin{proof}
One direction is trivial, since a Souslin tree, considered as a forcing notion, is $c.c.c.$ and adds no new reals.

For the other direction suppose there is a non-trivial $c.c.c.$ forcing notion $\mathbb{P}$ which adds no new reals. Let $\mathbb{B}=R.O(\mathbb{P})$ be the boolean completion of $\mathbb{P}$. $\mathbb{B}$ is a $c.c.c.$ complete Boolean algebra which is $(\omega, \omega)$-distributive, hence it is in fact $(\omega, \infty)$-distributive, thus it is a Souslin algebra, which implies the existence of a Souslin tree (see \cite{jech}). Hence SH fails.
\end{proof}
\begin{remark}
The above theorem is also proved by Shelah \cite{shelah2}, by completely different methods.
\end{remark}
As a corollary of the above theorem, and results of Solovay-Tennenbaum \cite{solovay-tennenbaum} (see also \cite{shelah1}), we have the following.
\begin{corollary}
It is consistent that any non-trivial $c.c.c.$ forcing notion adds a new real.
\end{corollary}

\section{Consistently, for every uncountable $\kappa,$ every $\kappa$-closed forcing of  size $2^{<\kappa}$ collapses some cardinal}
In this section, we consider statement $(2),$ and prove the following consistency result.
\begin{theorem}
Assuming the existence of an $\aleph_{\kappa^{++}}$-strong cardinal $\kappa$, it is consistent that for all uncountable cardinals $\lambda,$ any non-trivial $\lambda$-closed forcing notion of size $\leq 2^{<\lambda}$ collapses some cardinal.
\end{theorem}
\begin{remark}
The conclusion of the theorem implies $GCH$ fails everywhere, so some very large cardinals are needed for the theorem.
\end{remark}
\begin{proof}
To prove the theorem, we need two lemmas, which are of some independent interest.

\begin{lemma}\label{lem: singular continuum collapsing} if $2^{\kappa}$ is singular, then every non-trivial $\kappa^+$-closed forcing of size $2^{\kappa}$ collapses $2^{\kappa}$.
\end{lemma}
\begin{proof}
Let $\mathbb{P}$ be a non-trivial $\kappa^+$-closed forcing notion of size $2^{\kappa}$. Then forcing with $\mathbb{P}$ adds a new sequence $\tau$ of ordinals of size $\lambda < 2^{\kappa}$ (since the minimal such $\lambda$ must be regular). We will encode $2^\kappa$ into $\tau$.

If we choose $\lambda$ to be minimal, every initial segment of $\tau$ is in $V$. Now, we can define a function $F\in V$ from all possible initial segments of $\tau$ onto $^{\kappa}2$, such that for every $p\in \mathbb{P}$ and every $x\in$~$^{\kappa}2$ there is $q\leq p$ and $\beta < \lambda$ such that $q\Vdash \tau \restriction \beta = \check{a}$ and $F(a)=x$.

Let $\langle (p_i, x_i) : i < 2^{\kappa}\rangle$ enumerate $\mathbb{P}\times$~$ ^{\kappa}2$. For every $\alpha < 2^{\kappa}$, we use the $\kappa^+$-closure of $\mathbb{P}$ and the fact that $\tau \notin V$ in order to find $q \leq p_\alpha$ such that $q \neq p_i$ for every $i < \alpha$, $q\Vdash \tau \restriction \beta = \check{a}$ (for some $\beta$) and $F(a)$ is not determined yet, and set $F(a)=x_\alpha$.

Let $p = p_\alpha \in \mathbb{P}$. We start by building a tree of $2^{<\kappa}$ incompatible conditions $q_s,\, s\in$~$^{<\kappa}2$ such that $q_\emptyset = p$ and for every $s\in$~$^{<\kappa}2$, $q_{s^\smallfrown (0)}, q_{s^\smallfrown(1)} \leq q_s$, there is $\beta_s < \lambda$ such that $q_{s^\smallfrown(i)} \Vdash \tau \restriction \check\beta_s = \check{a}_{s^\smallfrown(i)}$ for $i\in\{0,1\}$, $a_{s^\smallfrown(0)} \neq a_{s ^\smallfrown(1)}$. This is possible since $\tau \notin V$ but every initial segment of it is in $V$.

For every $f\in$~$^{\kappa}2$, let us pick a condition $q_f \in \mathbb{P}$ such that for all $\alpha<\kappa, q_f \leq q_{f\restriction \alpha}$ (this is possible by the closure of $\mathbb{P}$). Let $\beta_f = \sup_{\alpha < \kappa} \beta_{f\restriction \alpha}$. Then $\forall f \in$~$^{\kappa}2,$ $q_f\Vdash \tau \restriction \check\beta_f = \check{a_f}$, where $a_f = \bigcup_{\alpha < \kappa} a_{f\restriction \alpha}$, and for every $f\neq f^\prime$, $a_f \neq a_{f^\prime}$. Since we chose already only $|\alpha |<2^{\kappa}$ values for $F$, there must be some $f\in$~$^{\kappa}2$ such that $q_f \neq p_i$ for every $i < \alpha$ and $F(a_f )$ was not already defined.

At the end of this process, there might be still possible initial segment of $\tau$ such that $F$ is undefined on, so we define $F(x)$ to be arbitrary on those values.

By density arguments, in $V[G]$, $\{F(\tau \restriction \beta) : \beta < \lambda\} = (2^\kappa)^V$.
\end{proof}
We can slightly generalize the lemma, and show that every $\kappa^+$-strategically closed forcing of size $2^\kappa$ collapses a cardinal. The argument is the same, and the only difference is in the construction of the function $F$. There, in the $\alpha$-th step, when we build the tree of extensions of $p_\a,$ we use the strategy in order to ensure that in limit stages of the tree we can always find $q_s$, stronger than $q_{s\restriction \alpha}$ for every $\alpha < \dom(s)$. This means that when picking $q_{s^\smallfrown(\epsilon)}$ for $\epsilon \in \{0,1\}$ we first extend $q_s$ into two incompatible conditions that force different information about $\tau$ as we did in the former case, and then extend those conditions to $q_{s^\smallfrown(0)}, q_{s^\smallfrown(1)}$ according to the strategy, assuming that the last step was made by the bad player. The limit stages are completely defined by the strategy, since the good player plays those steps.

The lemma also holds in a slightly more general setting. For every regular cardinal $\kappa$, if $2^{<\kappa}$ is singular then any $\kappa$-closed forcing of cardinality $2^{<\kappa}$ collapses some cardinal. The proof is essentially the same.

We also need the following.
\begin{lemma}\label{lem: singular continuum everywhere}
Assume $GCH$ holds and $\kappa$ is $\aleph_{\kappa^{++}}$-strong. Then there is a generic extension of the universe in which $\kappa$ remains inaccessible and for all infinite cardinals $\lambda<\kappa, 2^{\l}$ is a singular cardinal.
\end{lemma}
\begin{proof}
We use the extender based Radin forcing as developed in \cite{merimovich}, and continued in \cite{friedman-golshani}. Our presentation follows \cite{friedman-golshani}. We assume the reader is familiar with these papers and use the definitions and results from them without any mention.

Let $V^*$ denote the ground model. Let $j\colon V^* \to M^*$ be an elementary embedding witnessing the $\aleph_{\k^{++}}$-strongness of $\k$ and let $\bar{E}\in V^*$ be an extender sequence system derived from $j$, $\bar{E}=\langle \bar{E}_\a: \a\in \dom(\bar{E}) \rangle$, where $\dom(\bar{E})=[\k, \aleph_{\k^{+}})$ and $\len(\bar{E})=\k^+.$ Then the ultrapower $j_{\bar{E}}\colon V^* \to M^*_{\bar{E}}\simeq Ult(V^*,\bar{E})$ has critical point $\k$ and $M^*_{\bar{E}}$ contains $V^*_{\aleph_{\k^{+}}}$. Consider the following elementary embeddings $\forall \tau' < \tau < \len(\bar{E})$
\begin{align*} \label{E-system}
& j_\gt\func  \VS \to \MSt \simeq \Ult(\VS, E(\gt))=
\{j_\tau(f)(\bar E_\alpha \restricted \tau)\mid f\in V^*, \alpha\in [\kappa, \aleph_{\kappa^+})\},
\notag \\
&  k_\gt(j_\gt(f)(\Es_\ga \restricted \gt))=
        j(f)(\Es_\ga \restricted \gt),
\\
\notag & i_{\gt', \gt}(j_{\gt'}(f)(\Es_\ga \restricted \gt')) =
    j_\gt(f)(\Es_\ga \restricted \gt'),
\\
\notag & \ordered{\MSE,i_{\gt, \Es}} = \limdir \ordered {
        \ordof{\MSt} {\gt < \len(\Es)},
                \ordof{i_{\gt',\gt}} {\gt' \leq \gt < \len(\Es)}
        }.
\end{align*}

We demand that
        $\Es \restricted \gt \in \MSt$ for all $\tau<\len(\bar E)$.

Thus we get the following commutative diagram.

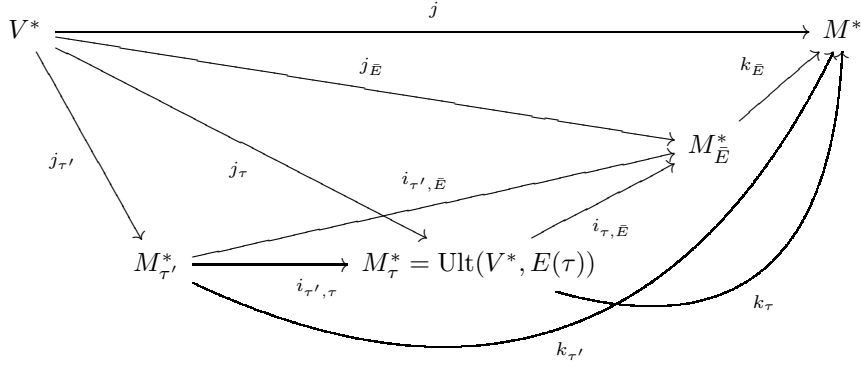
\[
\xymatrix{
\VS \ar[rrrrr]^j \ar[rrrrd]^{j_{\Es}} \ar[rrrdd]_{j_\gt} \ar[rdd]_{j_{\gt'}} & & & & & \MS \\
& & & & \MSE\ar[ur]^{k_\Es} & \\
& M^*_{\gt'}\ar@/_6pc/[uurrrr]_{k_{\gt'}} \ar[urrr]^{i_{\gt', \Es}} \ar[rr]_{i_{\gt', \gt}} &  & \MSt = \Ult(\VS, E(\gt)) \ar[ru]_{i_{\gt, \Es}}\ar@/_5pc/[rruu]_{k_\gt} & & \\
}
\]

Also factor through the normal ultrafilter to get the following commutative diagram

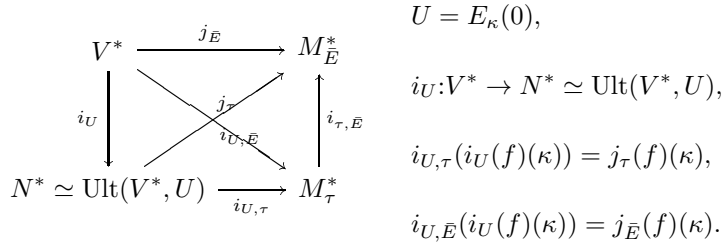
\begin{align*}
\begin{aligned}
\begin{diagram}
\node{\VS}
        \arrow{e,t}{j_\Es}
        \arrow{se,t}{j_\gt}
        \arrow{s,l}{i_U}
        \node{\MSE}
\\
\node{\NS \simeq \Ult(\VS, U)}
         \arrow{e,b}{i_{U, \gt}}
         \arrow{ne,b}{i_{U, \Es}}
        \node{\MSt}
         \arrow{n,b}{i_{\gt, \Es}}
\end{diagram}
\end{aligned}
\begin{aligned}
\qquad
\begin{split}
& U = E_\gk(0),
\\
& i_U \func  \VS \to \NS \simeq \Ult(\VS, U),
\\
& i_{U, \gt}(i_U(f)(\gk)) = j_\gt(f)(\gk),
\\
& i_{U, \Es}(i_U(f)(\gk)) = j_\Es(f)(\gk).
\end{split}
\end{aligned}
\end{align*}

Force with
\begin{center}
$\mathbb{R}=\Add(\k^+, (\aleph_{\k^{++}})^{M^*_{\bar{E}}})\times \Add(\k^{++}, (\aleph_{\k^{+3}})^{M^*_{\bar{E}}})\times \Add(\k^{+3}, (\aleph_{j_{\bar{E}}(\k)^+})^{M^*_{\bar{E}}}).$
\end{center}
Let $G$ be $\mathbb{R}-$generic over $V$.
By essentially the same arguments as those given in \cite{friedman-golshani}, \cite{merimovich},
we can find filters $G_{\bar{E}}, G_U$ and $G_\tau, \tau<\len(\bar{E}),$ such that $G_{\bar{E}}$
is $j_{\bar{E}}(\PP)$-generic over $M^*_{\bar{E}},$ $G_U$ is $i_U(\PP)$-generic
over $N^*$ and $G_\tau$ is $j_\tau(\PP)$-generic over $M^*_\tau$
and such that
 the following  diagram is commutative:

 \begin{align*}
\begin{diagram}
\node{V = \VS[G]}
        \arrow[2]{e,t}{j_\Es}
        \arrow{s,l}{i_{U}}
        \arrow{se,b}{j_{\gt'}}
        \arrow{see,b}{j_\gt}
    \node{}
        \node{\ME = \MSE[G_\Es]}
\\
    \node{N = \NS[G_U]}
         \arrow{e,b}{i_{U, \gt'}}
    \node{M_{\gt'} = M^*_{\gt'}[G_{\gt'}]}
         \arrow{ne,t,3}{i_{\gt', \Es}}
         \arrow{e,b}{i_{\gt', \gt}}
    \node{\Mt = \MSt[G_\gt]}
        \arrow[1]{n,b}{i_{\gt, \Es}}
\end{diagram}
\end{align*}
Set
\begin{center}
$\mathbb{R}_U=\Big(\Add(\k^+, \aleph_{\k^{++}})\times \Add(\k^{++}, \aleph_{\k^{+3}})\times \Add(\k^{+3}, \aleph_{i_{U}(\k)^+})\Big)^{N^*}.$

$\mathbb{R}_\tau=\Big(\Add(\k^+, \aleph_{\k^{++}})\times \Add(\k^{++}, \aleph_{\k^{+3}})\times \Add(\k^{+3}, \aleph_{j_{\tau}(\k)^+})\Big)^{M_\tau}.$

$\mathbb{R}_{\bar{E}}=\Big(\Add(\k^+, \aleph_{\k^{++}})\times \Add(\k^{++}, \aleph_{\k^{+3}})\times \Add(\k^{+3}, \aleph_{j_{\bar{E}}(\k)^+})\Big)^{M^*_{\bar{E}}}.$
\end{center}

Then we can find  $I_{U}, I_{\tau}$ and $I_{\bar{E}}$ in $V=V^*[G]$ such that:
\begin{enumerate}
\item $I_{U}$ is $\mathbb{R}_{U}$-generic over $N^{*}[G_{U}]$,

\item $I_{\tau}$ is $\mathbb{R}_{\tau}$-generic over $M_{\tau}^{*}[G_{\tau}]$,

\item $I_{\bar{E}}$ is $\mathbb{R}_{\bar{E}}$-generic over $M_{\bar{E}}^{*}[G_{\bar{E}}]$,

\item The generics are so that we have the following lifting diagram

\begin{align*}
\begin{diagram}
\node{}
    \node{}
        \node{\ME[I_\Es]}
\\
    \node{N[I_U]}
         \arrow{e,b}{i^*_{U, \gt'}}
    \node{M_{\gt'}[I_{\gt'}]}
         \arrow{ne,t,3}{i^*_{\gt', \Es}}
         \arrow{e,b}{i^*_{\gt', \gt}}
    \node{\Mt[I_\gt]}
        \arrow[1]{n,b}{i^*_{\gt, \Es}}
\end{diagram}
\end{align*}
\end{enumerate}

Iterate $j_\Es$ and consider the following diagram:
\begin{align*}
\begin{diagram}
\node{V}
        \arrow[2]{e,t}{j_\Es = j^{0,1}_\Es}
        \arrow{se,t,1}{j_{\gt_1}}
        \arrow{s,l}{i_{U}}
    \node{}
        \node{\ME}
        \arrow[2]{e,t}{j^{1,2}_\Es}
        \arrow{se,t,1}{j^2_{\gt_2}}
        \arrow{s,l}{i^2_{U}}
    \node{}
        \node{M_\Es^2}
        \arrow[2]{e,t}{j^{2,3}_\Es}
        \arrow{se,t,1}{j^3_{\gt_3}}
        \arrow{s,l}{i^3_{U}}
    \node{}
        \node{M_\Es^3}
        \arrow[1]{e,..}
\\
\node{N}
        \arrow[1]{e,b}{i_{U, \gt_1}}
        \arrow[1]{nee,t,3}{i_{U, \Es}}
    \node{M_{\gt_1}}
        \arrow[1]{ne,b,1}{i_{\gt_1, \Es}}
    \node{N^2}
        \arrow[1]{e,b}{i^2_{U, \gt_2}}
        \arrow[1]{nee,t,3}{i^2_{U, \Es}}
    \node{M^{2}_{\gt_2}}
        \arrow[1]{ne,b,1}{i^2_{\gt_2, \Es}}
    \node{N^3}
        \arrow[1]{e,b}{i^3_{U, \gt_3}}
        \arrow[1]{nee,t,3}{i^3_{U, \Es}}
    \node{M^{3}_{\gt_3}}
        \arrow[1]{ne,b,1}{i^3_{\gt_3, \Es}}
\end{diagram}
\end{align*}
where
\begin{align*}
& j^0_\Es = \id,
\\
& j^{n}_\Es = j^{0, n}_\Es,
\\
& j^{m, n}_\Es = j^{n-1, n}_\Es \circ \dotsb \circ j^{m+1, m+2}_\Es \circ j^{m ,m+1}_\Es.
\end{align*}

Let $R(-,-)$ be a function such that
\begin{center}
$i_{U}^{2}(R)(\kappa, i_{U}(\kappa))=\mathbb{R}_U,$

\end{center}
where $i_{U}^{2}$ is the second iterate of $i_{U}.$ It is also clear that
\begin{center}
 $j_{\bar{E}}^{2}(R)(\kappa, j_{\bar{E}}(\kappa))=\mathbb{R}_{\bar{E}}.$
\end{center}

As in \cite{merimovich} (see also \cite{friedman-golshani}), in the prepared model $V=V^*[G]$,  we define a new extender sequence system $\bar{F}= \langle \bar{F}_{\alpha}: \alpha \in \dom(\bar{F})\rangle$ by:
\begin{itemize}
  \item $\dom(\bar{F})=\dom(\bar{E}),$ \item $\len(\bar{F})=\len(\bar{E})$ \item $\leq_{\bar{F}}=\leq_{\bar{E}},$ \item $F(0)=E(0),$ \item $I(\tau)=I_{\tau},$ \item $\forall 0< \tau < \len(\bar{F}), F(\tau)= \langle \langle F_{\alpha}(\tau): \alpha \in \dom(\bar{F}) \rangle, \langle \pi_{\beta, \alpha}: \beta, \alpha \in \dom(\bar{F}), \beta \geq_{\bar{F}} \alpha \rangle \rangle$  is such that
    \begin{center}
    $X \in F_{\alpha}(\tau) \Leftrightarrow \langle \alpha, F(0), I(0), ..., F(\tau^{'}), I(\tau^{'}), ...: \tau^{'}  < \tau \rangle \in j_{\bar{E}}(X),$
    \end{center}
and
\begin{center}
$\pi_{\beta, \alpha}(\langle \xi, d \rangle)= \langle \pi_{\beta, \alpha}(\xi), d \rangle, $
\end{center}
\item $\forall \alpha \in \dom(\bar{F}), \bar{F}_{\alpha}= \langle \alpha, F(\tau),I(\tau): \tau < \len(\bar{F})  \rangle.$
\end{itemize}
Also let $I(\bar{F})$ be the filter generated by $\bigcup_{\tau < \len(\bar{F})} i_{\tau, \bar{E}}^{''}I(\tau).$ Then $I(\bar{F})$ is $\mathbb{R}_{\bar{F}}$-generic over $M_{\bar{F}}.$

%From now on we work with this new definition of extender sequence system and use $\bar{E}$ to
%denote it.

Working in $V$, let $\mathbb{P}_{\bar{F}}$ be the corresponding extender based Radin forcing, as defined in \cite{merimovich} (see also \cite[Definition 5.1]{friedman-golshani}). Note that the definition of $\mathbb{P}_{\bar{F}}$ depends on the function $R$, and this is were the iterability of the extender sequence plays a role. Let $H$ be $\mathbb{P}_{\bar{F}}$-generic over $V$. By reflection, we may assume that each $\bar{\mu}$ which appears in some condition in $\mathbb{P}_{\bar{F}}$ has $\dom(\bar{\mu})=[\k^0(\bar{\mu}), \aleph_{\k^0(\bar{\mu})^{+}}).$ For $\alpha \in \dom(\bar{F})$ set
\begin{center}
$C_{H}^{\alpha} = \{ \max\kappa(p_{0}^{\bar{F}_{\alpha}}): p \in H \}.$
\end{center}

The following clauses can be proved as in \cite{friedman-golshani}, \cite{merimovich}.
\begin{enumerate}
\item $V[H]$ and $V$ have the same cardinals.

\item $\kappa$ remains strongly inaccessible in $V[H]$.

\item $C_{H}^{\kappa}$ is a club in $\kappa.$

\item Let $\k_0=\min(C_{H}^{\kappa}),$ and let $K$ be $\Add(\omega, \aleph_{\k_0^{+}})_{V[H]}$-generic over $V[H].$ Then the following hold in $V[H][K]:$
\begin{itemize}
\item $\forall \l <\k_0, 2^\l=\aleph_{\k_0^+}$.
\item If $\l < \l_*$ are two successive points in $C_{H}^{\kappa}$, then $2^{\l^{+n}}=\aleph_{\l^{+n+1}},$ for $n=0,1,2$ and $2^{\l^{+3}}=\aleph_{\l_*^+}$.
\end{itemize}

\end{enumerate}
In particular $V[H][K] \models $``$\k$ is strongly inaccessible and for all $\l<\k,$ $2^\l$ is a singular cardinal''. The lemma follows.
\end{proof}
We are now ready to complete the proof of Theorem 3.1.

Let $V$ be a model of $GCH+\k$ is $\aleph_{\kappa^{++}}$ strong. By Lemma~\ref{lem: singular continuum everywhere}, there exists a cardinal preserving generic extension $V$ of $V^*$ in which $\k$ remains inaccessible, and for all infinite cardinals $\l<\k, 2^\l$ is a singular cardinal. Let $\k_*\leq \k$ be the least inaccessible cardinal of $V$, and consider the model $V_{\k_*}.$ It is a model of $ZFC$. We show that $V_{\k_*}$ is as required. So let $\l<\k_*$ be an uncountable cardinal, and let $\mathbb{P}\in V_{\k_*}$ be a non-trivial $\l$-closed forcing notion of size $2^{<\l}.$ As for a singular cardinal $\l,$ being $\l$-closed implies $\l^+$-closed, we can assume without loss of generality that $\l$ is a regular cardinal, and hence by our choice of $\k_*, \l=\mu^+,$ for some cardinal $\mu.$ So it follows from Lemma~\ref{lem: singular continuum collapsing} that forcing with $\mathbb{P}$ collapses $2^\mu$. The theorem follows.
\end{proof}
\section{consistency strength of statement $(2)$ for forcing notions of arbitrary size}
In this section we discuss the consistency of statement $(2)$ for forcing notions of arbitrary large size. We show that in this case the problem is much more difficult, and it requires some very large cardinals. For the sake of simplicity, let's just consider the case $\kappa=\aleph_0.$

The following shows that non-triviality of the forcing is essential. Suppose $\lambda>\aleph_1$ is regular and let $\mathbb{P}=\lambda$, as set of ordinals. Order $\mathbb{P}$ with the reversed ordinal order separated at $\omega_1$, i.e. $p \leq_\mathbb{P} q$ iff $p \geq q$ as ordinals and $p,q < \omega_1$ or $\omega_1 \leq p,q$.
It is clear that $\mathbb{P}$ is $\aleph_1$-closed (but not $\aleph_2$-closed), and it preserves all cardinals.

\begin{theorem}
Let $\l$ be a regular cardinal such that $2^{<\l} = \l$. For every regular cardinal $\mu < \l$ there is a separative forcing notion $\mathbb{P}$ which is $\mu$-closed, not $\mu^+$-closed and $\l$-distributive. Moreover, $|\mathbb{P}| = \l$, so $\mathbb{P}$ does not collapse cardinals.
\end{theorem}
\begin{proof}
The forcing $\mathbb{P}$ will be the forcing that adds a non-reflecting stationary subset of $S^{\l}_{<\mu} = \{\alpha < \l \mid \cf\alpha < \mu\}$, by  bounded conditions. Let us describe $\mathbb{P}$ precisely:

The conditions in $\mathbb{P}$ are functions $p \colon S^{\l}_{<\mu} \cap \alpha \to 2$ such that $\alpha < \l$ and $p^{-1}(1)$ is non-stationary at every $\beta \leq \alpha$ such that $\cf\beta\geq\mu$.

The order of $\mathbb{P}$ is end extension. Let us show that $\mathbb{P}$ has the required properties.

First, since $2^{<\l} = \l$, there are exactly $\l$ bounded subsets of $\l$. In particular, $|\mathbb{P}| \leq \l$. On the other hand, for every $\alpha \in S^{\l}_{<\mu}$, the function $p\colon S^{\l}_\mu \cap \alpha  \to 2$ defined by $p(\b) = 0$ for every $\b$ is a condition in $\mathbb{P}$, so $|\mathbb{P}| = \l$.

Next, $\mathbb{P}$ is $\mu$-closed since if $\langle p_\alpha \mid \alpha < \eta\rangle$, $\eta < \mu$, is a decreasing sequence of conditions, then $p = \bigcup_{\alpha<\eta} p_\alpha$ is a condition in $\mathbb{P}$. The only thing that we need to verify is that $p^{-1}(1)$ is non-stationary, or that $\cf(\dom(p)) < \mu$, and clearly if the sequence is not eventually constant the second option occurs.

Let us show that $\mathbb{P}$ is not $\mu^+$-closed. For each $\alpha< \mu$ let $p_\alpha \in \mathbb{P}$ be the condition with $\dom(p_\alpha) = \alpha$ and $p_\alpha (\b) = 1$ for every $\b$. Then the sequence of all $p_\alpha$, $\alpha < \mu$ has no lower bound, since for any condition $q$ such that $q\leq p_\alpha$ for every $\alpha$ we would have $\alpha \subseteq q^{-1}(1)$ for every $\alpha < \mu$ and therefore $\mu \subseteq q^{-1} (1)$, so $q^{-1}(1)$ is stationary at $\mu$. Similar argument shows that  we can construct such sequences below any condition $p \in \mathbb{P}$, so $\mathbb{P}$ is nowhere $\mu^+$-closed.

The main property of $\mathbb{P}$ is its $\l$-distributiveness. Let us show that $\mathbb{P}$ is $\l$-strategically closed, and in particular $\l$-distributive.

Let $\langle p_\alpha \mid \alpha < \eta\rangle$ be the play until step $\eta$. If $\eta$ is a successor ordinal, the strategy of the good player is to pick some ordinal in $S^{\l}_{<\mu}$ above the supremum of the domain of $p_{\eta - 1}$ and extend $p_{\eta - 1}$ up to this ordinal by appending zeros. If $\eta$ is limit less than $\l$ and $\bigcup_{\alpha < \eta} p_\alpha$ is a condition (so the good player did not lose already), the strategy will be to set $p_\eta = \bigcup_{\alpha < \eta} p_\alpha \cup \{\langle\eta,0  \rangle\}$, if $\eta \in S^{\l}_{<\mu}$, and otherwise just $p_\eta = \bigcup_{\alpha < \eta} p_\alpha$.

Let us prove that this is indeed a winning strategy, namely that at every limit stage of the play below $\l$, the union of the conditions until this step is a condition. Let $\eta$ be a limit ordinal of cofinality at least $\mu$. Then $\{\supp(p_\alpha) : \alpha < \eta,\,\text{limit}\}$ is a club at $\sup_{\alpha < \eta}\supp (p_\alpha)$ that witnesses that $\bigcup p_\alpha ^{-1} (1)$ is non-stationary, as wanted.

Since $\mathbb{P}$ is $\l$-distributive, it doesn't collapse cardinals $\leq \l$. Moreover, since $|\mathbb{P}|=\l$, $\mathbb{P}$ is $\l^{+}$-$c.c.$ and therefore it does not collapse cardinals above $\l$, so $\mathbb{P}$ does not collapse cardinals at all.
\end{proof}
\begin{corollary}
Assume there is an uncountable regular cardinal $\l$ such that $2^{<\l}=\l.$ Then there is a non-trivial $\aleph_1$-closed (but not $\aleph_2$-closed) forcing notion which preserves all cardinals.
\end{corollary}

\begin{corollary}
Suppose there is no inner model with a measurable cardinal $\kappa$ of Mitchell order $\kappa^{++}.$ Then there is a non-trivial $\aleph_1$-closed but not $\aleph_2$-closed forcing notion which preserves all cardinals.
\end{corollary}
\begin{proof}
By results of Gitik and Mitchell \cite{mitchell}, there is a strong limit singular cardinal $\kappa$ such that $2^\kappa=\kappa^+$. Applying the previous lemma with $\mu=\aleph_1$ we get the desired result.
\end{proof}
It follows that if we want a model in which all non-trivial $\aleph_1$-closed forcing notions collapse some cardinals, then $GCH$ should fail everywhere in that model.

Let us close this section by showing that there exists a non-trivial $\aleph_1$-closed cardinal preserving forcing notion in the model $V[H][K]$ of section 3. We assume that $V^*$ (and hence $V=V^*[G]$) has no inner model with a Woodin cardinal. It then follows that the combinatoric principle $\Box_\mu$ holds for every singular $\mu$ (see \cite{schimmerling-zeman}).

Let $\mu=\aleph_{\k_0^+}$ and $\l=\mu^+.$ Note that $V[H][K]=V[K][H],$ and that
\begin{center}
$V[K]\models$``$2^{\aleph_0}=\mu + \Box_\mu + SCH$ holds at $\mu$''.
\end{center}
It follows from \cite{rinot} that there exists  an $\aleph_1$-closed $\l$-Souslin tree $T$ in $V[K]$. We show that $T$ remains an $\aleph_1$-closed $\l$-Souslin tree in $V[K][H].$

To do this, we need a finer analysis of the model $V[K][H].$ For inaccessible cardinals $\a<\b$ set
\begin{center}
$\mathbb{P}(\a,\b)=(\Add(\a^+, \aleph_{\a^{++}})\times \Add(\a^{++}, \aleph_{\a^{+3}})\times Add(\a^{+3}, \aleph_{\b^+}))^V.$
\end{center}
Then it is easily seen that
\begin{center}
$V[K][H]=V[K\times H_1][H_2]=V[K][H_1][H_2],$
\end{center}
where
\begin{enumerate}
\item $K\times H_1$ is $\Add(\aleph_0, \mu)\times\mathbb{P}(\k_0, \k_1)$-generic over $V,$ where $\k_0 < \k_1$ are the first two elements of the club $C_H^\k.$
\item $V_{\k_1}^{V[K][H]}=V_{\k_1}^{V[K][H_1]}.$
\end{enumerate}
Now $V\models$``$\Add(\aleph_0, \mu)\times\mathbb{P}(\k_0, \k_1)$ is $\k_0^{+4}$-Knaster'', hence  $V[K]\models$``$\mathbb{P}(\k_0, \k_1)$ is $\k_0^{+4}$-Knaster, and $\k_0^+$-distributive''.
So after forcing with $\mathbb{P}(\k_0, \k_1)$ over $V[K]$, $T$ remains an $\aleph_1$-closed $\l$-Souslin tree. So by clause $(2)$ above
\begin{center}
$V[K][H]\models$`` $T$ is an $\aleph_1$-closed $\l$-Souslin tree''.
\end{center}
Thus $T$, considered as a forcing notion, is a non-trivial $\aleph_1$-closed forcing notion of size $\l$ which preserves all cardinals.

\section{Adding club sets in the absence of $GCH$}
In this section, we continue our study of the last section, and present a different method for producing cardinal preserving forcing notions in the absence of instances of $GCH$. Our forcing notions will be of size of a regular cardinal $\l$ such that $\l^{<\l} = \l$. Given any regular cardinal $\kappa$ with $\kappa^+ < \lambda,$ note that there is a natural forcing (i.e., $\Add(\l,1)$) which is $\kappa^+$-closed and $\l$-distributive. However the forcing $\Add(\l, 1)$
is also $\kappa^{++}$-closed and it preserves stationary subsets of $\lambda.$ The next theorem shows that it is consistent that there are
 $\kappa^+$-closed but not $\kappa^{++}$-closed forcing notions of size $\lambda$ which are $\l$-distributive  and   destroy a stationary subset of $\lambda$.

The results of this section are joint work with Moti Gitik, and are presented here with his kind permission.

\begin{theorem}  Assume $GCH$. Let $\kappa$ and $\lambda>\kappa^+$ be  regular cardinals. Then there is a cofinality preserving generic extension in which the following hold:  $2^\kappa=\lambda$ and there exists a non-trivial $\kappa^+$-closed, but not $\kappa^{++}$-closed, forcing notion of size $\lambda$ which preserves all cardinals (in fact the forcing is $\lambda$-distributive, $\lambda^+$-$c.c.$).
\end{theorem}
\begin{proof}
Let $\mathbb{P}_1=Add(\lambda, 1)$ and let $G_1$ be $\mathbb{P}_1$-generic over $V$. Let $F=\bigcup G_1.$ Then $F: \lambda \rightarrow 2.$ Let $\mathbb{P}_2=Add(\kappa, \lambda)^{V[G_1]}$ and let $G_2$ be $\mathbb{P}_2$-generic over $V[G_1].$ In $V[G_1*G_2], 2^\kappa=\lambda.$ We show that in $V[G_1*G_2],$ there is a non-trivial $\kappa^+$-closed, but not $\kappa^{++}$-closed, forcing notion of size $\lambda$ which preserves all cardinals. Let
\begin{center}
$S=\{ \alpha <\lambda : cf(\alpha)\leq \kappa$ or $(cf(\alpha)>\kappa, F(\alpha)=1) \}.$
\end{center}
$S$ is easily seen to be a stationary co-stationary subset of $\lambda$ in models  $V[G_1]$ and $V[G_1*G_2]$. Let $\mathbb{R} \in V[G_1*G_2]$ be the forcing notion for adding a club into $S$ by  approximations of cardinality $<\lambda.$
\begin{claim}
$(a)$ $\mathbb{R}$ is $\lambda^+$-$c.c.$

$(b)$ $\mathbb{R}$ is $\kappa^+$-closed but not $\kappa^{++}$-closed.

$(c)$ $\mathbb{R}$ is $\lambda$-distributive.
\end{claim}
\begin{proof}
$(a)$ and $(b)$ are trivial, so let's prove $(c).$ Let $K$ be $\mathbb{R}$-generic over $V[G_1*G_2]$,
 and let $h\in V[G_1*G_2*K], h:\mu \rightarrow On,$ where $\mu<\l$ is regular. We show that $h\in V[G_1*G_2].$

 Work in $V[G_1 * G_2]$. Let $F=\bigcup G_1: \lambda \rightarrow 2$ be as above and  let $\theta$ be large enough regular. By density arguments, we can find a continuous chain $\langle M_\a: \a<\mu  \rangle$ of elementary submodels of $H(\theta)$ such that:
\begin{enumerate}
\item Every initial segment of $\langle M_\a: \a<\mu  \rangle$ is in $V$,
\item $|M_\a| <\l,$
\item $\a<\b \Rightarrow M_\a \subseteq M_\beta,$
\item $\langle  M_\a: \a<\beta  \rangle\in M_{\beta+1},$
\item $\forall \a<\mu, M_\a\cap \l\in \l,$
\item $\forall \a<\mu, F(M_\a\cap \l)=1,$
\item $F(M\cap \l)=1,$ where $M=\bigcup_{\a<\mu}M_\a.$
\end{enumerate}
We may define the sequence $\langle M_\a: \a<\mu  \rangle$ in such a way that
$M$ contains all the relevant information. Let $\delta=M\cap\l$. Using the $\lambda$-closure of $Add(\lambda,1)$ and the chain condition of $Add(\kappa,\lambda)$, we can easily arrange $G_1\upharpoonright \delta=G_1\cap \Add(\delta, 1)^M$ is $\Add(\delta, 1)^M$-generic over $M$ and $\{\a<\delta: F(\a)=1 \}$ contains a club $C$ in $\delta$ (e.g. $C = \{ M_\a \cap \l \mid \alpha < \delta\}$). By our construction we may also assume that all the initial segments of $C$ are in $M[G_1\upharpoonright\delta].$ Let $G_2\upharpoonright \delta=G_2\cap \Add(\k,\delta)^{M[G_1\upharpoonright\delta]}.$ Then $G_2\upharpoonright\delta$ is
$\Add(\k,\delta)^{M[G_1\upharpoonright\delta]}$-generic over $M[G_1\upharpoonright\delta].$

Consider the model $N=M[G_1\upharpoonright\delta *G_2\upharpoonright\delta].$ Now we can decide values of $h$ inside $N$ and use $C$ to ensure that the conditions of $\mathbb{Q}$
used in the process go up to $\delta$. This will allow us to extend them finally to a single
condition deciding all the values of $h$. It follows that every initial segment of $h$ is in $N$ and hence $h\in V[G_1*G_2],$ as required.
\end{proof}
It follows that $\mathbb{R}$ preserves all the cardinals, and the theorem follows.
\end{proof}

Although we shoot a club through the fat stationary set $S = \{ \alpha < \kappa \mid F(\alpha) = 1\}$, we can't use Abraham-Shelah's general theorem from \cite{abraham-shelah}, since their cardinal arithmetic assumptions do not hold.

We now present a generalization of the above theorem, whose proof is essentially the same.
\begin{theorem}
Let $\k$ and $\l>\k^+$ be regular cardinals, and let $V[G_1*G_2]$ be a generic extension of $V$ by $\Add(\l,1)*\lusim{\Add}(\k, \l).$ Suppose $\mathbb{Q}\in V[G_1*G_2]$ is a cardinal preserving forcing notion of size $<\l,$ and let $H$ be $\mathbb{Q}$-generic over $V[G_1*G_2]$. Let $S=\{  \alpha <\lambda : cf(\alpha)\leq \kappa$ or $(cf(\alpha)>\kappa, \bigcup G_1(\alpha)=1)  \}$. Then $S$ is a stationary subset of $\l$ in $V[G_1*G_2*H],$ and if $\mathbb{R}$ is the forcing notion for adding a club subset of $S$, using  approximations of size $<\l,$ then
\begin{center}
$V[G_1*G_2*H]\models$``$\mathbb{R}$ is a $\k^+$-closed cardinal preserving forcing notion''.
\end{center}
\end{theorem}
As an application of the above theorem, let us prove an analogue of Theorem 5.1 for $\k$ singular.
\begin{theorem}
Suppose $\kappa$ is a strong cardinal and $\lambda>\kappa^+$ is regular. Then there is a forcing extension in which the following hold: $\kappa$ is a strong limit singular cardinal, $2^\kappa=\lambda$ and there exists a non-trivial $\kappa^+$-closed, but not $\kappa^{++}$-closed, cardinal preserving forcing notion of size $\lambda$.
\end{theorem}
\begin{proof}
By results of Gitik-Shelah and Woodin \cite{gitik-shelah}, we can assume that $\k$ is indestructible under the following kind of forcing notions:
\begin{itemize}
\item $\Add(\k, \delta),$ for any ordinal $\delta,$
\item Any $\k^+$-weakly closed forcing notion which satisfies the Prikry condition.
\end{itemize}
Fix any regular cardinal $\delta<\k.$ Force with $\Add(\l,1)*\lusim{\Add}(\k, \l)$ and let $G_1*G_2$ be $\Add(\l,1)*\lusim{\Add}(\k, \l)$-generic over $V$. Then $\k$ remains strong in $V[G_1*G_2].$ Let $\mathbb{Q}$ be the Prikry or Magidor forcing for changing the cofinality of $\k$ into $\delta$ and let $H$ be $\mathbb{Q}$-generic over $V[G_1*G_2].$ Then by Theorem 5.3, we can find, in $V[G_1*G_2*H],$  a $\k^+$-closed, but not $\k^{++}$-closed cardinal preserving forcing notion.
\end{proof}

\section{Consistently, for all $\k, 2^\kappa>\k^+$ and there is a $\kappa^+$-closed cardinal preserving forcing notion}
In this section, we again consider statement $(2),$ and prove a global consistency result, which is, in some sense, in the opposite direction of Theorem 3.1.
\begin{theorem}
Assuming the existence of a strong cardinal and infinitely many inaccessible cardinals above it, it is consistent that $GCH$ fails everywhere, and for each regular uncountable cardinal $\k,$ there exists a non-trivial $\k$-closed forcing notion which preserves all cardinals.
\end{theorem}
\begin{proof}
To prove the theorem, we need two auxiliary lemmas.
\begin{lemma}
Assume $2^\k$ is weakly inaccessible, and for all $\k<\l<2^\k, 2^\l=2^\k.$ Then there exists a $\k^+$-closed forcing notion which preserves all cardinals.
\end{lemma}
\begin{proof}
By Lemma 4.1.
\end{proof}
The next lemma is a generalization of the main theorem of Foreman-Woodin \cite{foreman-woodin}, where the use of a supercompact cardinal with infinitely many inaccessible
cardinals  above it,  is replaced by the much weaker assumption of the existence of a strong cardinal with an inaccessible above it.
\begin{lemma}
Assume $GCH$ holds, $\k$ is a strong cardinal and there exists an inaccessible cardinal above $\k$. Then there is a generic extension in which $\k$ remains inaccessible, for all $\l<\k, 2^\l$ is weakly inaccessible and $\l<\mu<2^\l$ implies $2^\mu=2^\l.$
\end{lemma}
\begin{proof}
The proof is similar to the proof of Lemma 3.4, so we follow that proof and just mention the changes which are required.
Let $V^*$ denote the ground model.
Also, for an ordinal $\alpha$, let us denote the next inaccessible cardinal above $\alpha$ by $\inacc(\alpha)$.

Let $j: V^* \rightarrow M^*$ be an elementary embedding witnessing the $\inacc(\kappa)^+$-strongness of $\kappa$ and let $\bar{E} \in V^*$
be an extender sequence system derived from $j$, $\bar{E}=\langle \bar{E}_\a: \a\in \dom(\bar{E}) \rangle$, where $\dom(\bar{E})=[\k,  \inacc(\kappa))$ and $\len(\bar{E})=\k^+.$ Then the ultrapower $j_{\bar{E}}\colon V^* \to M^*_{\bar{E}}\simeq Ult(V^*,\bar{E})$ has critical point $\k$ and $M^*_{\bar{E}}$ contains $V^*_{\inacc(\kappa)}$. As before, consider the resulting elementary embeddings $j_{\bar{E}}, j_\tau, k_\tau, i_{\tau', \tau}$ and $i_{\tau, \bar{E}}$, where
$\tau' < \tau < \kappa^+.$
Also factor through the normal ultrafilter to get the normal measure $U$ and embeddings $i_U, i_{U, \tau}$
and $i_{U, \bar{E}}$.

Force with
\begin{center}
$\MPB=\Add(\inacc(\kappa), \inacc(\kappa)^{+3})$.
\end{center}
Let $G$ be $\MPB$-generic over $V^*$ and let $V=V^*[G]$. As before, we can find suitable
filters $G_{\bar{E}}, G_U$ and $G_\tau, \tau<\len(\bar{E}),$ such that $G_{\bar{E}}$
is $j_{\bar{E}}(\PP)$-generic over $M^*_{\bar{E}},$ $G_U$ is $i_U(\PP)$-generic
over $N^*$ and $G_\tau$ is $j_\tau(\PP)$-generic over $M^*_\tau$ and such that the resulting
diagram commutes.

Set
$\MPB_U= \Add(\inacc(\kappa), i_U(\kappa))^{N^*},$
and define the forcing notions $\MPB_\tau$
and $\MPB_{\bar{E}}$ similarly, where $i_U, N^*$ are replaced with
$j_\tau, M^*_\tau$ and $j_{\bar{E}}, M^*_{\bar{E}}$ respectively.
We show that there are
 $I_{U}, I_{\tau}$ and $I_{\bar{E}}$ in $V=V^*[G]$ such that
 $I_{U}$ is $\mathbb{P}_{U}$-generic over $N^{*}[G_{U}]$,
$I_{\tau}$ is $\mathbb{P}_{\tau}$-generic over $M_{\tau}^{*}[G_{\tau}]$, $I_{\bar{E}}$ is $\mathbb{P}_{\bar{E}}$-generic over $M_{\bar{E}}^{*}[G_{\bar{E}}]$, and the generics are so that the corresponding diagram lifts.

Set $I_{\bar{E}}=G \cap \MPB_{\bar{E}}.$ We show that it is $\MPB_{\bar{E}}$-generic over $M_{\bar{E}}^{*}[G_{\bar{E}}]$. Using Easton's lemma, it suffices to
show genericity over $M_{\bar{E}}^{*}.$ Let $A \subseteq \mathbb{P}_{\bar{E}}$ be a maximal antichain in $M_{\bar{E}}^{*}$
and $X= \bigcup \{\dom(p): p \in A    \}.$ Then $|X| \leq \inacc(\kappa),$ and $A$ is a maximal antichain of $\Add(\inacc(\kappa), X)^{M_{\bar{E}}^{*}}$.
As $\Add(\inacc(\kappa), X)^{M_{\bar{E}}^{*}}=\Add(\inacc(\kappa), X)^{V^*}$, $A$ is a maximal antichain of $\Add(\inacc(\kappa), X)^{V^*},$  hence a maximal antichain
of $\Add(\inacc(\kappa), \inacc(\kappa)^{+3})^{V^*}$. Let $p \in A \cap G(\kappa).$ Then $p \in A \cap I_{\bar{E}}.$

Now set
\begin{center}
$I_U = \langle i_{U, \bar{E}}^{-1''}[I_{\bar{E}}]    \rangle,$ ~the filter generated by ~ $i_{U, \bar{E}}^{-1''}[I_{\bar{E}}]$
\end{center}
and
\begin{center}
$I_\tau = \langle i_{\tau, \bar{E}}^{-1''}[I_{\bar{E}}]    \rangle,$ ~the filter generated by ~ $i_{\tau, \bar{E}}^{-1''}[I_{\bar{E}}]$
\end{center}
It is easily seen that $I_U$ and $I_\tau$ are respectively $\MPB_U$-generic over $N^*$ and $\MPB_\tau$-generic over $M^*_\tau$. Now by applying Easton's lemma, we can conclude the desired result.

Let $R(-,-)$ be a function such that

\begin{center}
$i_{U}^{2}(\kappa, i_{U}(\kappa))=\mathbb{P}_U,$

\end{center}
where $i_{U}^{2}$ is the second iterate of $i_{U}.$ Working in $V=V^*[G],$ define the new extender sequence system
$\bar{F}= \langle \bar{F}_{\alpha}: \alpha \in \dom(\bar{F})\rangle$ as before.

Working in $V$, let $\mathbb{P}_{\bar{F}}$ be the corresponding extender based Radin forcing, and let $H$ be $\mathbb{P}_{\bar{F}}$-generic over $V$. By reflection, we may assume that each $\bar{\mu}$ which appears in some condition in $\mathbb{P}_{\bar{F}}$ has $\dom(\bar{\mu})=[\k^0(\bar{\mu}), \inacc(\k^0(\bar{\mu}))).$ For $\alpha \in \dom(\bar{F})$ set

\begin{center}
$C_{H}^{\alpha} = \{ \max\kappa(p_{0}^{\bar{F}_{\alpha}}): p \in H \}.$

\end{center}

The following clauses can be proved as before:

\begin{enumerate}
\item $V[H]$ and $V$ have the same cardinals.

\item $\kappa$ remains strongly inaccessible in $V[H]$

\item $C_{H}^{\kappa}$ is a club in $\kappa,$

\item If $\l < \l_* $ are two successive points in $C_{H}^{\kappa}$, then $2^\l = \inacc(\lambda)$
 and $2^{\inacc(\l)}=\l_*,$

\item Let $\gamma_0=\min(C_{H}^{\kappa}).$ If $\gamma_0\leq \l< \mu<2^\l<\k$, then $2^\mu=2^\l.$

\end{enumerate}
It follows immediately that
\begin{center}
$V[H]\models$`` If $\gamma_0\leq \l<\k,$ then $2^\l$ is weakly inaccessible and $\l< \mu<2^\l \Rightarrow 2^\mu=2^\l$''.
\end{center}
Force over $V[H]$ with $Add(\aleph_0, \gamma_0),$ and let $K$ be $Add(\aleph_0, \gamma_0)$-generic over $V[H].$ Then
\begin{center}
 $V[H][K]\models$``If  $\l<\k,$ then $2^\l$ is weakly inaccessible and $\l< \mu<2^\l \Rightarrow 2^\mu=2^\l.$
\end{center}
The lemma follows.
\end{proof}
We are now ready to complete the proof of Theorem 6.1.

Let $V^*$ be a model of $GCH+\k$ is a strong cardinal $+$ there are infinitely many inaccessible cardinals above $\k.$ By Lemma 6.3 there exists a generic extension $V$ of $V^*$ in which $\k$ remains measurable and for all infinite cardinals $\l<\k, 2^\l$ is weakly inaccessible and if $\l<\mu<2^\l,$ then $2^\mu=2^\l.$ Let $\k_*\leq \k$ be the least inaccessible of $V$, and consider the $ZFC$ model $V_{\k_*}.$ By Lemma 6.2, $V_{\k_*}$ is as required and the theorem follows.
\end{proof}
\begin{remark}
The use of an inaccessible cardinal above the strong cardinal is essentially to help us in constructing a generic extension in which the power function is such that for all infinite cardinals $\kappa, 2^\kappa$ is weakly inaccessible and if $\kappa <\lambda<2^\kappa$ then $2^\lambda=2^\kappa.$ It seems that we need such a behavior of the power function if we want to produce cardinal preserving forcing notions using Theorem 4.1, as otherwise the corresponding forcing will have a larger size and we will face some troubles in checking the chain condition.
\end{remark}

Note that in the above arguments, we always have $\k<\l<2^\k \Rightarrow 2^\l=2^\k,$ so it is natural to ask if it is consistent that $\k^+ < 2^\k < 2^{<2^\k}$ and there exists a $\k^+$-closed
( but not $\k^{++}$-closed) forcing notion which preserves all cardinal.
The next lemma gives a positive answer to this question.
\begin{lemma}
Suppose $GCH$ holds and $\k < \mu < \delta < \l$ are regular cardinals. Let $\mathbb{P}=\Add(\k, \delta),$ $\mathbb{Q}=\Add(\mu, \l)$ and let $G\times H$ be $\mathbb{P}\times \mathbb{Q}$-generic over $V$. Let $\mathbb{R}=\Add(\delta, 1)_{V[G]}.$ Then the following hold in $V[G \times H]:$

$(a)$ $\mathbb{R}$ is $\mu$-closed but not $\mu^+$-closed.

$(b)$ $\mathbb{R}$ is $\delta$-distributive and has size $\delta$.

In particular $V[G\times H]\models$`` $2^\k=\delta, 2^\mu=\l$ and there exists a $\mu$-closed (but not $\mu^+$-closed ) forcing notion of size $\delta=2^{<\mu}$ which preserves all cardinals''.
\end{lemma}
\begin{remark}
In fact our proof shows that it suffices to have:
\begin{itemize}
\item $\mathbb{P}$ is $\delta$-Knaster and it forces $2^{<\delta}=\delta,$
\item $\mathbb{Q}$ is $\mu$-closed and $\delta$-$c.c.$
\end{itemize}

\end{remark}
\begin{proof}
Recall that since $\mathbb{P}$ is $\kappa^+$-$c.c.$ and $\mathbb{Q}$ is $\mu$-closed, $\mu \geq \kappa^+$, by Easton's lemma,
\begin{center}
$\Vdash_{\mathbb{P} } \mathbb{Q}$ is
$\mu$-distributive, and
$\Vdash_{\mathbb{Q} } \mathbb{P}$ is $\kappa^+$-c.c.
\end{center}
Note that $V[G]\models \mathbb{R}$ is $\delta$-closed, and in particular $\mu$-closed. $\mathbb{Q}$ does not add any sequence of elements of $\mathbb{R}$ of length shorter than $\mu$, so $\mathbb{R}$ is still $\mu$-closed in $V[G][H]$.
Let us show now that $\mathbb{R}$ is $\delta$-distributive in $V[G][H]$ (and in particular it does not collapse cardinals).

For this end, we first show that $\mathbb{Q}$ is $\delta$-$c.c.$ after forcing with $\mathbb{P}\ast \lusim{\mathbb{R}}$. Assume otherwise, and let $\mathcal{A}\subseteq \mathbb{Q}$ be an antichain of cardinality $\delta$. Since $\mathbb{R}$ is $\delta$-closed in $V[G]$, one can decide the values of $\mathcal{A}$ by induction on $\alpha < \delta$ and obtain an antichain of cardinality $\delta$ in $V[G]$.
Let $\langle a_i \mid i < \delta \rangle$ enumerates $\mathcal{A}$ in $V[G]$ and let $p_i \Vdash$`` $\lusim{a}_i = \check{q_i}$'' by some condition $p_i$ in $\mathbb{P}$, where $\lusim{a}_i$ is a $\mathbb{P}$-name for $a_i$. Since $\mathbb{P}$ is $\delta$-Knaster there is $I\subset \delta$ such that for every $i, j \in I$, $p_i$ is compatible with $p_j$. Therefore, $q_i \perp q_j$ for every $i \neq j$ in $I$, and this implies that already in $V$ there is an antichain in $\mathbb{Q}$ of cardinality $\delta$, which is a contradiction.

Let $K$ be a $V[G][H]$-generic for $\mathbb{R}$ and let $x$ be a sequence of ordinals of length $<\delta$ in $V[G][H][K]$. Since $V[G][H][K] = V[G][K][H]$, and $\mathbb{Q}$ is $\delta$-$c.c.$ in $V[G][K]$, there is a set $\lusim{x}$ of cardinality $<\delta$ which is a $\mathbb{Q}$-name for $x$ (so $\lusim{x} \subseteq \mathbb{Q}\times On$).
Since $\mathbb{R}$ is $\delta$-distributive in $V[G]$, $\lusim{x}\in V[G]$. Therefore, $x\in V[G][H]$, as wanted.
\end{proof}

School of Mathematics, Institute for Research in Fundamental Sciences (IPM), P.O. Box:
19395-5746, Tehran-Iran.

E-mail address: golshani.m@gmail.com

The Hebrew University of Jerusalem, Einstein Institute of Mathematics,
Edmond J. Safra Campus, Givat Ram, Jerusalem 91904, Israel.

E-mail address:  yair.hayut@mail.huji.ac.il

\end{document}